\documentclass[12pt]{amsart}
\usepackage{graphicx}
\usepackage[centertags]{amsmath}
\usepackage{amsfonts}
\usepackage{amssymb}
\usepackage{amsthm}

\newtheorem{thm}{Theorem}

\newtheorem{lem}{Lemma}[section]
\newtheorem{cor}{Corollary}[section]

\newtheorem{prop}[lem]{Proposition}
\newtheorem*{thm*}{Theorem}
\theoremstyle{definition}

\theoremstyle{remark}
\newtheorem{rem}{Remark}[section]
\numberwithin{equation}{section}

\newcommand{\norm}[1]{\left\Vert#1\right\Vert}

\newcommand{\T}{\mathrm{T}}

\newcommand{\pd}[2]{\frac{\partial #1}{\partial #2}}

\newcommand{\vphi}{{\varphi}}

\newcommand{\calB}{\mathcal{B}}
\newcommand{\calC}{\mathcal{C}}

\newcommand{\calL}{\mathcal{L}}
\newcommand{\calM}{\mathcal{M}}

\newcommand{\calP}{\mathcal{P}}

\newcommand{\calT}{\mathcal{T}}

\newcommand{\bbZ}{\mathbb{Z}}

\newcommand{\bbR}{\mathbb R}
\newcommand{\bbC}{\mathbb C}

\newcommand{\bbN}{\mathbb N}

\newcommand{\bbH}{\mathbb{H}}

\newcommand{\Tr}{ \mbox{Tr}}

\newcommand{\SL}{ \mathrm{SL}}

\newcommand{\PSL}{ \mathrm{PSL}}
\newcommand{\PGL}{ \mathrm{PGL}}

\newcommand{\link}{\mathrm{link}}
\newcommand{\lk}{\mathrm{link}}
\newcommand{\Alt}{\mathrm{Alt}}

\newcommand{\bs}{\backslash}
\newcommand{\id}{1\!\!1}

\renewcommand{\Im}{\mathrm{Im}}

\newcommand{\bfa}{\mathbf{a}}

\begin{document}
\title[Linking numbers of modular knots]
{Quadratic irrationals and linking numbers of modular knots}%
\author{Dubi Kelmer}%
\address{Boston College, Department of Mathematics, Chestnut Hill, MA 02467}
\email{kelmer@bc.edu}

\thanks{This work was partially supported by NSF grant DMS-1237412.}%

\date{\today}%
\dedicatory{}%
\commby{}%

\begin{abstract}
A closed geodesic on the modular surface gives rise to a knot on the 3-sphere with a trefoil knot removed, and one can compute the linking number of such a knot with the trefoil knot. We show that, when ordered by their length, the set of closed geodesics having a prescribed linking number become equidistributed on average with respect to the Liouville measure. We show this by using the thermodynamic formalism to prove an equidistribution result for a corresponding set of quadratic irrationals on the unit interval.
\end{abstract}

\maketitle
\section{Introduction}
Let $\calM=\PSL_2(\bbZ)\bs \bbH$ denote the modular surface; here $\bbH$
denotes the upper half plane endowed with the hyperbolic metric, and $\PSL_2(\bbZ)$ acts by isometries on $\bbH$ via linear fractional transformation. The closed geodesics are the periodic orbits for the geodesic flow on the unit tangent bundle $T^1\calM$. Following the results of Ghys \cite{Ghys07} and Sarnak \cite{Sarnak10link} on the linking numbers of modular knots, we continue the study of the set of primitive closed geodesics having a prescribed linking number, and in particular their distribution on $T^1\calM$.

In \cite{Ghys07}, Ghys showed that $T^1\calM$ is homeomorphic to the three-sphere without a trefoil knot, and hence, a closed geodesic can be thought of as a knot in this space. He then computed the linking number of a such a knot with the trefoil knot in terms of a certain arithmetic invariant. If $\gamma$ denotes an oriented primitive closed geodesic in $T^1\calM$, we will denote this linking number by $\lk(\gamma)$ and we refer to it as the linking number of this primitive geodesic.

Motivated by this result, Sarnak \cite{Sarnak10link} indicates how a careful analysis of the Selberg trace formula for modular forms of real weights can be used to study the number of primitive closed geodesics with a prescribed linking number and bounded length (see also \cite{Mozzochi10}). In particular, if $\calC_n(T)$
denotes the set of primitive closed geodesics, $\gamma$, with linking number $\lk(\gamma)=n$ and length $\ell(\gamma)\leq T$, then his analysis implies that $|\calC_n(T)|\sim \frac{e^T}{3T^2}$ (that is $|\calC_n(T)|/\frac{e^T}{3T^2}\to 1$ as $T\to \infty$).
For comparison we recall that the \emph{prime geodesic theorem} states that $|\calC(T)|\sim \frac{e^T}{T}$, where $\calC(T)$ is the set of primitive closed geodesics with $\ell(\gamma)\leq T$.

In order to study the average distribution of these closed geodesics we adapt the approach of Pollicott \cite{Pollicott86}, using the correspondence between closed geodesics and quadratic irrationals together with the thermodynamic formalism for the Gauss map.

The Gauss map $\mathrm{T}:(0,1)\to [0,1)$ is given by $\mathrm{T}(x)=\{\frac{1}{x}\}$ where
where $\{\cdot\}$ denotes the fractional part. The iteration of this map gives a dynamical system whose periodic orbits are the quadratic irrationals with a periodic continued fractions, that is, the set
\[Q=\{x=[\overline{a_1,\ldots a_n}]=\tfrac{1}{a_1+\frac{1}{a_2+\ldots\frac{1}{a_n+\frac{1}{a_1+\ldots}}}}|n\in \bbN,\; a_1,\ldots a_n\in \bbN\}.\]

There is a well known correspondence, going back to Artin \cite{Artin24} and beautifully explained by Series \cite{Series85}, between the periodic orbits of the Gauss map and the primitive closed geodesics on the modular surface. This correspondence gives a natural ordering of the quadratic irrationals, that is, for $x\in Q$ we let $\ell(x)$ denote the length of the corresponding primitive closed geodesic and we order them according to this length.

Using Mayer's \cite{Mayer76} thermodynamic formalism for the Gauss map, in \cite{Pollicott86} Pollicott showed that, with this ordering, the quadratic irrationals become equidistributed on $(0,1)$ with respect to the Gauss measure given by $d\nu=\frac{1}{\log(2)}\frac{dx}{1+x}$. Then, using the correspondence between the Gauss map and the geodesic flow, he deduced the average equidistribution of the full set of closed geodesics when ordered by length.

\begin{rem}
There is another natural, and well known, correspondence between primitive closed geodesics on the modular surface and equivalence classes of binary quadratic forms (see \cite{Sarnak82}). Instead of ordering the geodesics by length one can order them according to their discriminant (that is, the discriminant of the corresponding quadratic form). In this setting,  Duke's theorem \cite{Duke88} shows that the set of closed geodesics of a given discriminant also become equidistributed in $T^1\calM$ as the discriminant goes to infinity. Since closed geodesics with the same discriminant also have the same length, Duke's theorem also implies that the set of closed geodesics with a fixed length become equidistributed when the length goes to infinity, which also implies the average equidistribution.
\end{rem}

The linking number of a closed geodesic can be computed (up to a sign) from the corresponding continued fraction expansion.
Specifically,
we will show below that if $x\in Q$ has an even minimal periodic expansion $x=[\overline{a_1,\ldots,a_{2n}}]$, then its $\T$-orbit corresponds to two primitive geodesics $\gamma,\bar\gamma$  (related by orientation reversal symmetry) having the same length
\[\ell(\gamma)=\ell(\bar\gamma)=-2\sum_{j=1}^{2n}\log(\T^jx):=\ell(x),\]
and opposite linking numbers given by
\[\lk(\gamma)=-\lk(\bar\gamma)=-a_1+a_2-\ldots+a_{2n}.\]
On the other hand, if $x\in Q$ has an odd minimal expansion $x=[\overline{a_1,\ldots,a_{n}}]$ (with $n$ odd), then it corresponds to a single primitive geodesic (invariant under orientation reversal symmetry) having length
\[\ell(\gamma)=-4\sum_{j=1}^{n}\log(\T^jx):=\ell(x),\]
and linking number zero.

To any quadratic irrational $x\in Q$ with a minimal even expansion $x=[\overline{a_1,\ldots,a_{2n}}]$ we attach the alternating sum
\begin{equation}\label{e:Alt}\mathrm{Alt}(x)=-a_1+a_2-\ldots+a_{2n}.\end{equation}
Combining the dynamical approach of \cite{Pollicott86} with the results from \cite{Sarnak10link} we show that the quadratic irrationals with a given alternating sum (up to a sign) become equidistributed on average with respect to the Gauss measure. Specifically, for any $n\in \bbZ$ let
\begin{equation}\label{e:Qn} Q_n(T)=\{x\in Q| \mathrm{Alt}(x)=n,\;\ell(x)\leq T\}\end{equation}
\begin{thm}\label{t:Qn}
$|Q_n(T)|\sim \frac{\log(2)}{\pi^2}\frac{e^T}{T}$ and
for any $f\in C([0,1])$
\begin{equation}\label{e:equiQn}\lim_{T\to\infty}\frac{1}{|Q_n^+(T)|}\sum_{x\in Q_n^+(T)}f(x)=\int_0^1 f(x)d\nu(x),\end{equation}
where $Q_n^+(T)=Q_n\cup Q_{-n}$.
\end{thm}

The set of quadratic irrationals $x\in Q$ with $\Alt(x)=0$ contains the set $Q_{\mathrm{odd}}$ of quadratic irrationals with an odd periodic expansion. We show that this (much smaller) set is also equidistributed.
\begin{thm}\label{t:Qodd}
Let $Q_{\mathrm{odd}}(T)$ denote the set of quadratic irrationals with an odd minimal expansion and $\ell(x)\leq T$.
Then $|Q_{\mathrm{odd}}(T)|\sim \frac{3\log(2)}{\pi^2}\frac{e^{T/2}}{T}$ and
for any $f\in C([0,1])$ we have
\begin{equation}\label{e:Qodd}\lim_{T\to\infty}\frac{1}{|Q_{\mathrm{odd}}(T)|}\sum_{x\in Q_{\mathrm{odd}}(T)}f(x)=\int_0^1 f(x)d\nu(x).\end{equation}
\end{thm}
\begin{rem}
We note that the reason that $Q_{\mathrm{odd}}(T)$ is much smaller than $Q_{\mathrm{even}}(T)$ is an artifact of our ordering.
In a slightly different ordering given by $\tilde{\ell}(x)=-2\sum_{j=1}^n\log(\T^jx)$
(which is more natural when considering only the minimal expansion $x=[\overline{a_1,\ldots,a_n}]$) we get that roughly half the points are odd, half are even, and each is equidistributed.
\end{rem}

These equidistribution results of quadratic irrationals imply the equidistribution of the corresponding sets of closed geodesics.
Given a test function $F\in C_c(T^1\calM)$ and a closed geodesic $\gamma$ we denote by
\[\int_\gamma F=\int_0^{\ell(\gamma)} F(\vphi_t(x))dt,\]
where $\vphi_t:T^1\calM\to T^1\calM$ is the geodesic flow and $x\in T^1\calM$ is any point on the geodesic.
As a consequence of Theorem \ref{t:Qn} we get
\begin{thm}\label{t:main}
$|\calC_n(T)|\sim \frac{e^{T}}{3T^2}$ and
for any $F\in C_c(T^1\calM)$
\[\lim_{T\to \infty}\frac{1}{|\calC^+_{n}(T)|}\sum_{\gamma\in \calC^+_{n}(T)} \frac{1}{\ell(\gamma)}\int_\gamma F= \int_{T^1\calM}Fd\mu,\]
where $\calC_n^+(T)=\calC_n(T)\cup\calC_{-n}(T)$.
\end{thm}
\begin{rem}
It is likely that the same equidistribution result holds for the smaller set $\calC_n(T)$ instead of $\calC_n^+(T)$.
Note that the time reversal of a primitive geodesic has the same length and an opposite linking number. Since a closed geodesic and its time reversal have the same projection to the base manifold $\calM$, our result implies that the closed geodesics in $\calC_n(T)$ at least become equidistributed on the base manifold.
\end{rem}
\begin{rem}
In addition to the thermodynamic formalism, another crucial ingredient in the proof of Theorem \ref{t:Qn} (and hence also Theorem \ref{t:main}) is the estimate  \eqref{e:SM} proved in \cite{Mozzochi10,Sarnak10link} using the trace formula for modular forms of real weight. Consequently, our result does not give an independent proof for the asymptotics $|\calC_n(T)|\sim \frac{e^{T}}{3T^2}$. It would be interesting if one could obtain such a result using an entirely dynamical approach.
\end{rem}

\begin{rem}
 One should compare Theorem \ref{t:main} to analogous results on the counting and equidistribution of closed geodesics on a compact hyperbolic surface lying in a prescribed homology class. In this case, Katsuda and Sunada \cite{KatsudaSunada90}, Lalley \cite{Lalley89}, and Pollicott \cite{Pollicott91} obtained similar results using an entirely dynamical approach (which works also in variable negative curvature). The asymptotics for the number of closed geodesics in a fixed homology class were previously obtained by Phillips and Sarnak \cite{PhillipsSarnak87} using the Selberg trace formula (for compact hyperbolic surfaces). In \cite{Zelditch89Trace} Zelditch generalized the trace formula to give another proof of the equidistribution of closed geodesics in a fixed homology class (with explicit bounds on the rate of equidistribution). It should be possible to further generalize Zelditch's approach to give another proof for the equidistribution of closed geodesics with a prescribed linking number (such an approach would also give bounds for the rate of equidistribution).
\end{rem}

Let $\calC_i(T)\subset\calC_0(T)$ denote the subset of \emph{inert} geodesics, that is, the primitive geodesics that are left invariant under orientation reversal symmetry. From Theorem \ref{t:Qodd} we get
\begin{thm}\label{t:inert}
$|\calC_i(T)|\sim \frac{e^{T/2}}{T}$ and for any $F\in C_c(T^1\calM)$
\[\lim_{T\to\infty}\frac{1}{|\calC_{i}(T)|}\sum_{\gamma\in \calC_{i}(T)} \frac{1}{\ell(\gamma)}\int_\gamma F= \int_{T^1\calM}Fd\mu.\]
\end{thm}

\begin{rem}
The result in Theorem \ref{t:inert} is not new. In \cite{Sarnak07}, using the Selberg trace formula for $\PGL_2(\bbZ)$, Sarnak showed that $|\calC_i(T)|\sim \frac{e^{T/2}}{T}$. He also showed there, that if a primitive geodesic is inert, then all other primitive geodesics with the same discriminant are also inert. Consequently, the average equidistribution of the inert geodesics already follows from Duke's Theorem. We note however that the proof we give here is entirely dynamical and does not rely on the Selberg trace formula nor on Duke's theorem.
\end{rem}

\section{Background and notation}
We write $A\lesssim B$ (or $A=O(B)$) to indicate that $A\leq cB$ (or $|A|\leq c|B|$)
for some constant $c$. We use the notation $A(T)\sim B(T)$ to indicate that $A(T)/B(T)\to 1$ as $T\to \infty$.

\subsection{The modular surface}
We denote by $\bbH=\{z=x+iy:y>0\}$
the upper half plane. With the identification $T_z\bbH\cong\bbC$, the hyperbolic metric on $T_z\bbH$ is given by $\langle \xi,\eta\rangle_z=\frac{\Re(\xi\bar{\eta})}{y^2}$, and the unit tangent bundle is $T^1\bbH=\{(z,\xi)\in \bbH\times \bbC :|\xi|=\Im(z)\}$.
In this model the geodesics are either semi-circles orthogonal to the real line or vertical lines.

The group $\PSL_2(\bbR)$ acts on the hyperbolic plane by isometries via linear fractional transformations, that is, $g=\left(\begin{smallmatrix}a & b\\ c & d \end{smallmatrix} \right)$ sends $(z,\xi)\in T^1\bbH$ to
$$(g(z),g'(z)\xi)=(\frac{az+b}{cz+d},\frac{\xi}{(cz+d)^2}).$$
We can thus identify $T^1\bbH\cong \PSL_2(\bbR)$ and the unit tangent bundle of the modular surface $\calM=\PSL_2(\bbZ)\bs \bbH$ as
 $T^1\calM\cong \PSL_2(\bbZ)\bs \PSL_2(\bbR)$. With these identification the Liouville measure $\mu$ on $T^1\calM$ is the projection of the Haar measure of $\PSL_2(\bbR)$. This measure projects down to the hyperbolic area on $\calM$ normalized so that $\mathrm{Area}(\calM)=1$. Specifically, in the coordinates $(z,\xi)=(z,ye^{i\theta})$ we have $d\mu(x,y,\theta)=\frac{3}{\pi^2}\frac{dxdy}{y^2}d\theta$.

\subsection{Geodesic flow}
The geodesic flow $\vphi_t:T^1\calM\to T^1\calM$ sends a point $(z_0,\xi_0)\in T^1\calM$ to the point $(z_t,\xi_t)$ obtained by flowing for time $t$ along the geodesic $\gamma$ with starting point $\gamma(0)=z_0$ and direction $\gamma'(0)=\xi_0$. The closed geodesics on the modular surface are the closed orbits of this flow.

There is a one-one correspondence between oriented (primitive) closed geodesics on $T^1\calM$ and (primitive) hyperbolic conjugacy classes in $\PSL_2(\bbZ)$, and we denote by $\{A_\gamma\}$ the conjugacy corresponding to a closed geodesic $\gamma$. Here a closed geodesic is called primitive if it wraps once around and a hyperbolic element is primitive if it is not the power of some other hyperbolic element.

Recall that a hyperbolic element $A\in \PSL_2(\bbZ)$ has two fixed points on the real line (which are conjugate quadratic irrationals). We note that if $A\in \{A_\gamma\}$ then $\gamma$ has a lift $\hat{\gamma}$ to the upper half plane hitting the real line at the two fixed points of $A$ (different lifts will correspond to different conjugates of $A_\gamma$).

\subsection{Symmetries}
We have two natural symmetries on $T^1\calM$. One is \emph{time-reversal} symmetry given by  $(z,\xi)\mapsto (z,-\xi)$, and the other is \emph{orientation-reversal} symmetry given by $(z,\xi)\mapsto (-\bar{z},-\bar{\xi})$.

If $(z,\xi)\in T^1\calM$ is a point on a closed geodesic, $\gamma$, then $(z,-\xi)$ and $(-\bar{z},-\bar{\xi})$ are points on two different closed geodesic we call the time reversal, $\gamma^{-1}$, and the orientation reversal, $\bar{\gamma}$, of $\gamma$ respectively. If $\{A_\gamma\}$ is the conjugacy class in $\PSL_2(\bbZ)$ corresponding to $\gamma$, then $\{A_{\gamma^{-1}}\}=\{A_\gamma^{-1}\}$ and $\{A_{\bar{\gamma}}\}=\{wA_\gamma w\}$ with $w=\left(\begin{smallmatrix}1 & 0 \\ 0 & -1\end{smallmatrix}\right)$.

Following \cite{Sarnak07}, we call a primitive geodesic \emph{inert} if $\bar{\gamma}=\gamma$ and \emph{reciprocal} if $\gamma^{-1}=\gamma$. We note for future reference that $\gamma$ is inert if and only if $A_\gamma=B^2$ for some $B\in \PGL_2(\bbZ)$ with $\det(B)=-1$ (see \cite[Page 227]{Sarnak07}).

\subsection{Linking numbers}
We recall that a closed geodesic, $\gamma$, on $T^1\calM$ gives rise to a knot in the $3$-sphere with a trefoil knot removed, and that the linking number, $\lk(\gamma)$, is the linking number of this knot with the trefoil knot. In \cite{Ghys07}, Ghys showed that
$\lk(\gamma)=\Psi(A_\gamma)$, where $\Psi:\PSL_2(\bbZ)\to \bbZ$ denotes the Rademacher function.

The Rademacher function $\Psi(A)$ depends only on the conjugacy class of $A$, and can be computed by expressing $A$ as a word in the generators
\[S=\begin{pmatrix} 0 &-1\\ 1& 0\end{pmatrix},\quad U=\begin{pmatrix}  1& -1\\ 1&  0\end{pmatrix},\]
where we recall the presentation $\PSL_2(\bbZ)=<S,U|S^2=U^3=I>$.
The reader can consider the following as the definition of $\Psi(A)$ (see \cite[Page 54]{RademacherGrosswald72} for another equivalent definition and some properties of the Rademacher function).
Any element $A\in \PSL_2(\bbZ)$ is conjugated to either $S,U,U^{-1}$ or
\[SU^{\epsilon_1}\cdots SU^{\epsilon_k},\quad \epsilon_j=\pm 1.\]
We then have $\Psi(S)=0$, $\Psi(U^{\pm 1})=\mp 2 $ and $\Psi(A)=\sum_{j=1}^k\epsilon_j$ if
$A$ is conjugated to $SU^{\epsilon_1}\cdots SU^{\epsilon_k}$.

From this, together with the relation $wSU=SU^{-1}w$, we get that the linking number changes sign under time reversal or orientation reversal,
that is $\lk(\bar\gamma)=\lk(\gamma^{-1})=-\lk(\gamma)$. In particular, if $\gamma$ is inert or reciprocal then $\link(\gamma)=0$.

The Rademacher function also comes up in the multiplier system of the Dedekind
eta function
\begin{equation*}
\eta_D(\tau)=e^{\pi i \tau/12}\prod_{m=1}^\infty(1-e^{2\pi i m\tau}).
\end{equation*}
Specifically, for $A=\left(\begin{smallmatrix} a & b\\ c & d\end{smallmatrix}\right)\in\SL(2,\bbZ)$ with $\Tr(A)>0$
$$\eta_D(Az)=\nu_{1/2}(A)(cz+d)^{1/2}\eta_D(z),$$
with $\nu_{1/2}(A)=e^{i\pi\Psi(A)/12}$. Consequently, for any real $r\in (-6,6)$ we have that $\nu_r(A)=e^{i\pi r\Psi(A)/6}$ is a multiplier system for modular forms of real weight $r$.
Using this observation, in \cite{Sarnak10link} Sarnak shows how a careful analysis of the Selberg trace formula for modular forms of real weight (see \cite[Chapter 9]{Hejhal83}) implies
\begin{equation}\label{e:SM}
\sum_{\gamma\in \calC(T)} \ell(\gamma)e^{\frac{\pi i r }{6}\Psi(A_\gamma)}=
\left\lbrace\begin{array}{cc} \frac{e^{(1-\frac{|r|}{2})T}}{1-|r|/2}+ O(e^{\frac{3T}{4}}\log(\tfrac{1}{|r|})) & |r|\leq \frac{1}{2} \\
O(e^{\frac{3T}{4}}) & |r|>\frac{1}{2}\end{array}\right..
\end{equation}
Some of the delicate estimates needed for the proof were done by Mozzochi in \cite{Mozzochi10}.
From this, the estimate $|\calC_n(T)|\sim \frac{e^T}{3T^2}$ (and much more) is obtained by integrating \eqref{e:SM} against $e^{-\pi i n r/6}$.

\begin{rem}We note for future reference that \eqref{e:SM} still holds if we replace the sum over primitive geodesics by a sum over all closed geodesics.
This is because the prime geodesic theorem implies that the contribution of the non-primitive geodesics to such a sum is bounded by $O(Te^{T/2})$.
\end{rem}

\section{Closed geodesics and quadratic irrationals}
In this section we recall the results of Series \cite{Series85} on coding the geodesic flow as a suspended flow over the Gauss map, and use it to reduce the problem of equidistribution of closed geodesics to the equidistribution of quadratic irrationals.

\subsection{The Gauss map and continued fraction}
Any $x\in(0,1)$ has a continued fraction expansion
\[x=[a_1,a_2,a_3,\ldots]=\tfrac{1}{a_1+\frac{1}{a_2+\frac{1}{a_3+\ldots}}},\;a_1,a_2,a_3,\ldots\in \bbN,\]
(which terminates if and only if $x$ is rational). When the expansion is periodic, that is, if there is $n\in \bbN$ such that $a_{n+j}=a_j$ for all $j$, we write $x=[\overline{a_1,\ldots,a_n}]$. The points with such a periodic expansion are precisely the quadratic irrationals $x\in (0,1)$ whose conjugate satisfies $\bar{x}<-1$. Here the conjugate $\bar{x}$ of a quadratic irrational is the second root of the same quadratic polynomial.

Let $\T:(0,1]\to [0,1)$ denote the Gauss map given by $\T(x)=\{1/x\}$ where $\{x\}$ denotes the fractional part of $x$.
The Gauss map acts on the continued fraction expansion by a shift to the left, hence, the periodic points for $\T$ are the points with a periodic expansion.

Let $\hat{\T}:(0,1]\times
(0,1)\to [0,1)\times (0,1)$ denote the extension of the Gauss map, given by $\hat{\T}(x,y)=(\{\frac{1}{x}\},\frac{1}{y+[1/x]})$.
The extended map $\hat{\T}$ acts on the continued fraction of both points together by a shift
\[\hat{\T}([a_0,a_1,\ldots],[a_{-1},a_{-2},\ldots])=([a_1,a_2,\ldots],[a_{0},a_{-1},\ldots]).\]
Hence $\hat{\T}$ is invertible and there is a correspondence between $\hat{\T}$-closed orbits and $\T$-closed orbits.
Specifically, the $\T$-orbit of $x=[\overline{a_1,\ldots,a_n}]$ corresponds to the $\hat{\T}$-orbit of $(x,-1/\bar{x})=([\overline{a_1,\ldots,a_n}],[\overline{a_n,\ldots,a_1}])$.

We recall that the Gauss measure $\nu$ on $[0,1]$, given by
\begin{equation}\label{e:nu}
d\nu(x)=\frac{1}{\log2}\frac{dx}{(1+x)},
\end{equation}
is the unique $\T$-invariant probability measure on $[0,1]$ equivalent to Lebesgue measure. We extend the Gauss measure to the measure $\tilde{\nu}$ on $[0,1]\times[0,1]$ given by
\begin{equation}\label{e:tnu}
d\tilde{\nu}(x,y)=\frac{1}{\log2}\frac{dxdy}{(1+xy)^2}.
\end{equation}
We note that $\tilde{\nu}$ is $\hat{\T}$-invariant, that its projection to each factor is $\nu$, and that this is the only measure satisfying these properties.

\subsection{Geodesic flow and suspended flow}
Consider the spaces
$$\Sigma=(0,1]\times (0,1)\times \bbZ/2\bbZ, \mbox{ and } \Sigma^*=\{(x,y,e)\in \Sigma|1/x\not\in \bbN\}.$$
We can further extend the Gauss map to the map $\hat{\T}:\Sigma^*\to\Sigma$ by $\hat{\T}(x,y,e):=(\hat{\T}(x,y),e+1)$.
In \cite[Theorems A and B]{Series85}, Series defined a cross section $X\subset T^1\calM$ for the geodesic flow, together with a map
\[p_0:\Sigma \to X,\]
satisfying the following properties:
\begin{enumerate}
\item $p_0$ is continuous, open, and surjective. It is also injective except that $p_0(1,1,0)=p_0(1,1,1)$.
\item A geodesic $\gamma$ starting from $p_0(x,y,e)$ has a lift $\hat{\gamma}$ to $\bbH$ with endpoints $(\hat{\gamma}(-\infty),\hat{\gamma}(\infty))$ given by $(-y,1/x)$ if $e=0$ and $(y, -1/x)$ if $e=1$.
\item Let $X^*=p_0(\Sigma^*)$, then the first return map $P:X^*\to X$ is well defined and satisfies
$P\circ p_0=p_0\circ \hat{\T}.$
\end{enumerate}

For any $\sigma\in \Sigma^*$, let $r(\sigma)$ denote the hyperbolic distance between $p_0(\sigma)$ and $p_0(\hat{\T}(\sigma))$. This was explicitly computed in \cite[Section 3.2]{Series85} and is given by
$$r(\sigma)=r_0(\sigma)+r_0(\hat{\T}(\sigma)) \mbox{ with } r_0(x,y,e)=-\tfrac{1}{2}\log(xy).$$
Let
\[\Sigma_r=\{(\sigma,t)|\sigma\in \Sigma,\; 0\leq t\leq r(\sigma)\},\]
where we identify $(\sigma,r(\sigma))=(\hat{\T}(\sigma),0)$, and define the suspended flow $\psi_t:\Sigma_r\to\Sigma_r$ by $\psi_t(\sigma,t_0)=(\sigma,t_0+t)$.

We can extend the map $p_0:\Sigma\to X$ to a map $p:\Sigma_r\to T^1\calM$ satisfying $p\circ \psi_t=\vphi_t\circ p$ by flowing along the geodesics, that is, we let
$p(\sigma,t)=\vphi_t( p_0(\sigma))$. The map $p:\Sigma_r\to T^1\calM$ gives local coordinates on $T^1\calM$ in which the Liouville measure is given by
$d\mu(x,y,t)=\frac{3}{\pi^2}\frac{dxdy}{(1+xy)^2}dt$ (see \cite[Section 3.1]{Series85}).

\subsection{Periodic orbits}
The above coding gives a period preserving correspondence between primitive closed geodesics and primitive closed $\psi_t$-orbits.
Moreover, the closed $\psi_t$-orbits are easily classified in terms of the periods of the Gauss map, that is, the purely periodic quadratic irrationals.

Specifically, to any purely periodic quadratic irrational $x\in (0,1)$ we have two closed $\psi_t$-orbits, the orbits of $(x,-1/\bar{x},0)$ and of $(x,-1/\bar{x},1)$ (clearly these only depends on the $\T$-orbit of $x$). We note that the two closed geodesics corresponding to the same $\T$-orbit are a pair, $\gamma,\bar{\gamma}$, of orientation reversed geodesics.
For $x\in Q$ with a minimal even expansion $x=[\overline{a_1,\ldots, a_{2n}}]$, the length of each of the corresponding closed geodesics is given by
$$\ell(x):=\sum_{j=1}^{2n} r(x,-1/\bar{x})=-2\sum_{j=1}^{2n}\log(\T x).$$
 (see \cite[Section 3.2]{Series85}).
\begin{rem}
In this formula it is important that the expansion $x=[\overline{a_1,\ldots, a_{2n}}]$ is the minimal even expansion of $x$, in the sense that there is no $n'<n$ with $x=[\overline{a_1,\ldots, a_{2n'}}]$.
If $x$ has an odd minimal expansion, $x=[\overline{a_1,\ldots, a_{n}}]$ with $n$ odd, the length is still computed using the minimal even expansion $x=[\overline{a_1,\ldots, a_{n},a_1,\ldots, a_{n}}]$.
\end{rem}

\subsection{Linking numbers}
The linking number of a closed geodesic can be computed from the minimal continued fraction expansion of the corresponding quadratic irrational (up to a sign). To do this we will use the following lemma.
\begin{lem}\label{l:Ba}
Let $x=[\overline{a_1,\ldots,a_n}]$ be a minimal expansion of a quadratic irrational. Then the cyclic group of matrices in $\PGL_2(\bbZ)$ fixing $x$ is generated by $B_{a_1}\cdots B_{a_n}$ with
$B_a=\begin{pmatrix} 0& 1\\ 1 & a\end{pmatrix}$.
\end{lem}
\begin{proof}
Let $B\in \PGL_2(\bbZ)$ be a matrix fixing $x$. If we have that $B=B_{b_1}\cdots B_{b_k}$ for some $k\in \bbN$ and $b_1,\ldots,b_k\in \bbN$ then
$$[\overline{a_1,\ldots,a_n}]=x=Bx=[b_1,\ldots,b_k,\overline{a_1,\ldots,a_n}],$$
which can only happen if $k=nm$ and $b_{nl+j}=a_{j}$ for all $0\leq \ell \leq m$ and $1\leq j\leq n$, so that
$B=(B_{a_1}\cdots B_{a_n})^m$. Consequently, we need to show that either $B$ or $B^{-1}$ can be expressed as such a product.

First, if one of the coefficients of $B$ is zero then it must be in one of the forms $\left(\begin{smallmatrix}0&\pm 1\\1 &d\end{smallmatrix}\right),\;\left(\begin{smallmatrix}\pm 1& 0\\c &1\end{smallmatrix}\right),\;\left(\begin{smallmatrix}\pm 1& b\\0 &1\end{smallmatrix}\right)$, or $\left(\begin{smallmatrix}a&\pm 1\\1 &0\end{smallmatrix}\right)$. Of these the only ones having two fixed points $x\in(0,1)$ and $\bar{x}\in(-\infty,-1)$ are $\left(\begin{smallmatrix}0& 1\\1 &d\end{smallmatrix}\right)=B_d$ with $d\in \bbN$ and $\left(\begin{smallmatrix}-a& 1\\1 &0\end{smallmatrix}\right)=B_{a}^{-1}$ with $a\in \bbN$. We may thus assume from now on that $B$ has non-zero coefficients.

Next, we show that either $\pm B$ or $\pm B^{-1}$ has positive coefficients.
Write $B=\left(\begin{smallmatrix}a& b\\c &d\end{smallmatrix}\right)$ and (replacing $B$ with $-B$ if necessary we may) assume $a>0$.  The equation $|ad-bc|=1$ implies that $ad$ and $bc$ have the same sign. If $d>0$ then $bc>0$ so either $b,c>0$ and all coefficients are positive, or $b,c<0$ in which case $B^{-1}=\left(\begin{smallmatrix}d&-b\\-c &a\end{smallmatrix}\right)$ has positive coefficients.
Finally if $d<0$ then $bc<0$ and it is easy to check that if $a,b>0$ and $c,d<0$ then $\left(\begin{smallmatrix}a& b\\c &d\end{smallmatrix}\right)$ can't have a fixed point in $(0,1)$ and if $a,c>0$ and $b,d<0$ then $\left(\begin{smallmatrix}a& b\\c &d\end{smallmatrix}\right)^{-1}=\left(\begin{smallmatrix}d& -b\\-c &a\end{smallmatrix}\right)$ can't have a fixed points in $(-\infty,-1)$.

From here on (replacing $B$ with $\pm B^{-1}$ if necessary) we assume that $B=\left(\begin{smallmatrix}a& b\\c &d\end{smallmatrix}\right)$ has positive coefficients. Hence there is some $b_1\in \bbN\cup \{0\}$ such that
$$B_{b_1}^{-1}B=\left(\begin{smallmatrix}c-b_1a& d-b_1b\\a &b\end{smallmatrix}\right)=\left(\begin{smallmatrix}a'& b'\\c' &d'\end{smallmatrix}\right),$$
with $a'\in[0,a)$ and $c',d'>0$. If $a'=0$ or $b'=0$ we stop. Otherwise, the equation $|a'd'-b'c'|=1 $ implies that $b'>0$, so we still have a matrix with positive coefficients. We can reiterate this until we get to a matrix with $a'=0$ or $b'=0$ (this will terminate after at most $a$ steps).
We thus see that there are $b_1,\ldots, b_k\in \bbN\cup\{0\}$ such that $(B_{b_1}\cdots B_{b_k})^{-1}B$ is either $\left(\begin{smallmatrix}0& \pm 1\\1 &d'\end{smallmatrix}\right)$ with $d'>0$ or $\left(\begin{smallmatrix}1& 0\\c' &1\end{smallmatrix}\right)$ with $c'>0$.

We can easily exclude the case of $\left(\begin{smallmatrix}1& 0\\c' &1\end{smallmatrix}\right)$ by considering the action on the fixed point $\bar{x}\in(-\infty,-1)$ (notice that $\left(\begin{smallmatrix}1& 0\\c' &1\end{smallmatrix}\right)\bar x=\frac{\bar{x}}{c'\bar{x}+1}>0$ and $(B_{b_1}\cdots B_{b_k})^{-1}B\bar{x}=(B_{b_1}\cdots B_{b_k})^{-1}\bar{x}<0$).
This just leaves us with the two cases of $B=B_{b_1}\cdots B_{b_k}\left(\begin{smallmatrix}0& \pm 1\\1 &d'\end{smallmatrix}\right)$. We note that $\left(\begin{smallmatrix}0& 1\\1 &a\end{smallmatrix}\right)\left(\begin{smallmatrix}0& 1\\1 &0\end{smallmatrix}\right)\left(\begin{smallmatrix}0& 1\\1 &b\end{smallmatrix}\right)=\left(\begin{smallmatrix}0& 1\\1 &a+b\end{smallmatrix}\right)$,
so we may assume that all $b_j\in \bbN$ except perhaps $b_k$ and $b_1$ which could also be zero.

In the first case, $B=B_{b_1}\cdots B_{b_k}B_{d'}$, we must have that $b_1\neq 0$ as well, because otherwise $B_{b_1}\cdots B_{b_k}B_{d'}x>1$ in contradiction to $x=Bx\in (0,1)$. If $b_k>0$ we are done, while if $b_k=0$ we can rewrite $B=B_{b_1}\cdots B_{b_{k-2}} B_{\tilde{b}_{k-1}}$ with $\tilde{b}_{k-1}=b_{k-1}+d'\in \bbN$ so we are also done.

In the second case, $B=B_{b_1}\cdots B_{b_k}\left(\begin{smallmatrix}0&  -1\\1 &d'\end{smallmatrix}\right)$, if $b_k\neq 0$
use the identity
$$\left(\begin{smallmatrix}0& 1\\1 &b_k\end{smallmatrix}\right)\left(\begin{smallmatrix}0& -1\\1 &d'\end{smallmatrix}\right)=\left(\begin{smallmatrix}0& 1\\1 &b_k-1\end{smallmatrix}\right)\left(\begin{smallmatrix}0& 1\\1 &1\end{smallmatrix}\right)\left(\begin{smallmatrix}0& 1\\1 &d'-1\end{smallmatrix}\right),$$
to reduce this to the first case of $B=B_{b_1}\cdots B_{b_{k-1}}B_{b_k-1}B_1B_{d'-1}$. If $b_k=0$ we can rewrite
$B=B_{b_1}\cdots B_{b_{k-2}}\left(\begin{smallmatrix}0& -1\\1 & \tilde{d}\end{smallmatrix}\right)$ with $\tilde{d}=d'-b_{k-1}$.
Since $b_{k-2}\geq 1$ we have that $(B_{b_1}\cdots B_{b_{k-2}})^{-1}B\bar{x}<-1$ so we must have that $\tilde{d}\geq 1$ which is the case we dealt with above.

\end{proof}
%
%

\begin{prop}
Let $x=[\overline{a_1,\ldots,a_{2n}}]$ be the minimal even expansion of a quadratic irrational, let $\gamma$ be one of the corresponding primitive closed geodesics, and let $\{A_\gamma\}$ denote the corresponding hyperbolic conjugacy class. Then
\[|\Psi(A_\gamma)|=|\sum_{j=1}^{2n}(-1)^ja_j|=|\Alt(x)|.\]
(recall that $\Psi(A_{\bar{\gamma}})=-\Psi(A_\gamma)$.)
\end{prop}
\begin{proof}
Let $A=B_{a_1}\cdots B_{a_{2n}}\in \PSL_2(\bbZ)$. Since $x=[\overline{a_1,\ldots,a_{2n}}]$ is the minimal even expansion, by Lemma \ref{l:Ba} we have that $A,A^{-1}$ are the unique primitive hyperbolic elements in $\PSL_2(\bbZ)$ fixing $x,\bar{x}$. Since (some conjugate of) $A_\gamma$ also has these fixed points then either  $A\in \{A_\gamma\}$ or $A^{-1}\in \{A_\gamma\}$, and hence $\Psi(A_\gamma)=\pm \Psi(A)$.

Let $V=US$ and note that $B_a=wSV^a=SV^{-a}w$ so that
\[A=SV^{-a_1}SV^{a_2}SV^{-a_3}\cdots SV^{a_{2n}}.\]
Using the relations
$SV^aS=(SU)^a$ and $V^{-a}=(SU^{-1})^a$, we can write
\[A=S[SU^{-1}]^{a_1}[SU]^{a_2}[SU^{-1}]^{a_3}\ldots[SU]^{a_{2n}} S^{-1},\]
to get that indeed $\Psi(A)=-a_1+a_2-a_3+\ldots+a_{2n}$.
\end{proof}

Note that if $x=[\overline{a_1,\ldots,a_n}]$ is a minimal odd expansion then its minimal even expansion is $x=[\overline{a_1,\ldots,a_n,a_1,\ldots, a_n}]$. In this case, if $A\in \PSL_2(\bbZ)$ is the primitive element fixing $x$ then the above argument shows that $\Psi(A)=0$. Consequently, a primitive geodesic corresponding to $x\in Q_{\mathrm{odd}}$ has linking number zero. We now show that these are precisely the inert primitive geodesics.
\begin{prop}\label{p:inert}
A primitive closed geodesic is inert if and only if it corresponds to the $\T$-orbit of some $x\in Q_{odd}$.
\end{prop}
\begin{proof}
Let $\gamma$ be a closed geodesic, $x=[\overline{a_1,\ldots,a_n}]$ a quadratic irrational in the corresponding $\T$-orbit, and $A\in \{A_\gamma\}$ the primitive hyperbolic element fixing $x$. We recall that $\gamma$ is inert if and only if $A=B^2$ for some $B\in \PGL_2(\bbZ)$ with $\det(B)=-1$.
In this case, since any $B\in \PGL_2(\bbZ)$ with $\det(B)=-1$ has a fixed point on the real line and any fixed point of $B$ is also a fixed point of $A$, the fixed points of $B$ are also $x$ and $\bar{x}$. Lemma \ref{l:Ba}, now implies that $B=B_{a_1}\cdots B_{a_n}$ and hence $\det(B)=(-1)^n$ so $n$ must be odd.
\end{proof}

\subsection{Equidistribution}
We can use the correspondence between closed geodesics and quadratic irrationals to relate equidistribution of closed geodesics on $T^1\calM$ to equidistribution of quadratic irrationals on $[0,1]$.

Let $\calC'\subseteq \calC$ denote any set of oriented primitive closed geodesics which is invariant under orientation reversal symmetry (that is, $\gamma\in \calC'$ if and only if $\bar\gamma\in \calC'$). Let $Q'\subseteq Q$ denote the corresponding set of quadratic irrationals (i.e., the set of points obtained as endpoints of lifts of geodesics from $\calC'$). Let $\calC'(T)$ (respectively $Q'(T)$) denote the set of $\gamma\in \calC'$ with $\ell(\gamma)\leq T$ (respectively $x\in Q'$ with $\ell(x)\leq T$).
Consider the counting function
$$\Lambda_{\calC'}(T)=\sum_{\gamma\in \calC'(T)}\ell(\gamma).$$
The following proposition reduces Theorems \ref{t:main} and Theorem \ref{t:inert} to Theorem \ref{t:Qn} and Theorem \ref{t:Qodd} respectively.
\begin{prop}\label{p:equi}
In the above setting, if for any $f\in C([0,1])$
\begin{equation}\label{e:equiQ}
\lim_{T\to\infty}\frac{1}{|Q'(T)|}\sum_{x\in Q'(T)} f(x)=\int_0^1f(x)d\nu,\end{equation}
then $\Lambda_{\calC'}(T)\sim \tfrac{\pi^2}{3\log(2)}|Q'(T)|$ and for any $F\in C_c(T^1\calM)$
\begin{equation}\label{e:equiC}
\lim
_{T\to\infty}\frac{1}{\Lambda_{\calC'}(T)}\sum_{\gamma\in \calC'(T)}\int_\gamma F=\int_{T^1\calM} Fd\mu.
\end{equation}
Under the additional assumption $T^\beta e^{cT}\lesssim |\calC'(T)|\lesssim T^\alpha e^{cT}$ for some $c>0$ and $\beta\leq \alpha<\beta+1$ we also have
\begin{equation}\label{e:equiC2}
\lim_{T\to\infty}\frac{1}{|\calC'(T)|}\sum_{\gamma\in \calC'(T)}\frac{1}{\ell(\gamma)}\int_\gamma F=\int_{T^1\calM}fd\mu,\end{equation}
\end{prop}

Before we go on with the proof we need an intermediate step, showing that the equidistribution of quadratic irrationals $x\in Q$ on $([0,1],\nu)$ is actually equivalent to the equidistribution of pairs $(x,-1/\bar{x})$ in $([0,1]\times [0,1],\tilde{\nu})$.
\begin{lem}\label{l:equi}
The following are equivalent:
\begin{enumerate}
\item For any $f\in C([0,1])$
\[\lim_{T\to\infty}\frac{1}{|Q'(T)|}\sum_{x\in Q'(T)}f(x)=\int_0^1f(x)d\nu.\]
\item For any $f\in C([0,1]\times[0,1])$
\[\lim_{T\to\infty}\frac{1}{|Q'(T)|}\sum_{x\in Q'(T)} f(x,-1/\bar{x})= \int_0^1\int_0^1 f(x,y)d\tilde{\nu}.\]
\end{enumerate}
\end{lem}
\begin{proof}
The fact that $(2)\Rightarrow (1)$ is obvious, we now show that $(1)\Rightarrow(2)$.

For $\mathbf{k}=(k_1,\ldots, k_n)$ consider the open interval
$$I_{\mathbf{k}}=\{x=[a_1,a_2,a_3,\ldots]|a_j=k_j,\; j=1,\ldots, n\},$$
and note that any continuous function on $[0,1]^2$ can be approximated by a linear combination of indicator functions of products $I_{\mathbf{k}}\times I_{\mathbf{k'}}$. (This follows from the fact that any open interval with rational endpoints can be written as a finite intersections of unions of sets of the form $I_{\mathbf{k}}$ and their complements).

It is thus sufficient to consider test functions of the form $f=\id_{I_{\mathbf{k}}\times I_{\mathbf{k}'}}$ for any $\mathbf{k}=(k_1,\ldots, k_n),\;\mathbf{k}'=(k_1,\ldots, k_{n'})$. That is, we need to show that
\[\lim_{T\to\infty}\frac{1}{|Q'(T)|}\sum_{x\in Q'(T)} \id_{I_{\mathbf{k}}\times I_{\mathbf{k}'}}(x,-1/\bar{x})= \tilde{\nu}(I_{\mathbf{k}}\times I_{\mathbf{k}'}).\]

Since $\hat{\T}$ acts on the continued fraction by a shift, for $f=\id_{I_{\mathbf{k}}\times I_{\mathbf{k}'}}$ we have that
$f\circ \hat{\T}^{n'}=\id_{I_{\mathbf{k''}}\times[0,1]}$ where $\mathbf{k}''=(k_{n'}',\ldots, k_1',k_1,\ldots, k_n)$.
On the other hand, since acting by $\hat{\T}$ only permutes the periodic $\T$-orbits (which all have the same length) and the set $Q'$ is by definition composed of complete $\T$-orbits, then any function $f$ on $[0,1]\times[0,1]$  satisfies
\[\frac{1}{|Q'(T)|}\sum_{x\in Q'(T)} f(x,-1/\bar{x})=\frac{1}{|Q'(T)|}\sum_{x\in Q'(T)} f\circ\hat{T}^n(x,-1/\bar{x}).\]
In particular for $f=\id_{I_{\mathbf{k}}\times I_{\mathbf{k}'}}$
\begin{eqnarray*}
\frac{1}{|Q'(T)|}\sum_{x\in Q'(T)} \id_{I_{\mathbf{k}}\times I_{\mathbf{k}'}}(x,-1/\bar{x})&=&\frac{1}{|Q'(T)|}\sum_{x\in Q'(T)} \id_{I_{\mathbf{k''}}\times[0,1]}(x,-1/\bar{x})\\
&=&\frac{1}{|Q'(T)|}\sum_{x\in Q'(T)} \id_{I_{\mathbf{k''}}}(x)
\end{eqnarray*}
We can now approximate $\id_{I_{\mathbf{k''}}}$ by continuous functions get that
\begin{eqnarray*}\lim_{T\to\infty}\frac{1}{|Q'(T)|}\sum_{x\in Q'(T)} \id_{I_{\mathbf{k''}}}(x)&=&\nu(I_{\mathbf{k''}})\\
&=&\tilde{\nu}(I_{\mathbf{k''}}\times [0,1])=\tilde{\nu}(I_{\mathbf{k}}\times I_{\mathbf{k'}}),
\end{eqnarray*}
where the last equality follows from the invariance of $\tilde{\nu}$ under $\hat{\T}$.
\end{proof}

\begin{proof}[Proof of Proposition \ref{p:equi}]
For the first part we follow Pollicott's argument from \cite[Page 436]{Pollicott86}.
To any continuous bounded $F\in C_c(T^1\calM)$ we let $f_F\in C((0,1)\times (0,1))$ be given by
\[f_F(x,y)=\int_0^{r(x,y)}(F(p(x,y,0,t))+F(p(x,y,1,t)))dt.\]
We then have
$$\int_{T^1\calM}Fd\mu=\frac{3}{\pi^2}\int_0^1\int_0^1 f_F(x,y)\frac{dxdy}{(1+xy)^2}=\frac{3\log(2)}{\pi^2}\int_0^1\int_0^1 f_Fd\tilde{\nu}$$
and for any pair $\gamma,\bar\gamma$ of orientation reversed geodesics
\[\int_\gamma F+\int_{\bar\gamma}F=\sum_{x\in Q_\gamma}f_F(x,-1/\bar{x}),\]
where $Q_\gamma\subset (0,1)$ is the set of endpoints (in $(0,1)$) of lifts of $\gamma$ to $\bbH$.

The result now follows from lemma \ref{l:equi} as
\begin{eqnarray*}\lim_{T\to\infty}\frac{1}{\Lambda_{\calC'}(T)}\sum_{\gamma\in \calC'(T)}\int_\gamma F&=&\lim_{T\to\infty}\frac{\sum_{x\in Q'(T)}f_F(x,-1/\bar{x})}{\sum_{x\in Q'(T)}f_1(x,-1/\bar{x})}\\ &=&\frac{\int_0^1\int_0^1 f_F d\tilde{\nu}}{\int_0^1\int_0^1 f_1 d\tilde{\nu}}=\int_{T_1\calM} Fd\mu.
\end{eqnarray*}
In particular, for $F=1$ we have $f_1(x,y)=2r(x,y)$ so that
\[\frac{\Lambda_{\calC'}(T)}{|Q'(T)|}=\frac{1}{|Q'(T)|}\sum_{x\in Q'(T)}r(x,-1/\bar{x}),\]
converges as $T\to\infty$ to
\[\int_0^1\int_0^1 2r(x,y)d\tilde{\nu}=-4\int_0^1\frac{\log(x)}{1+x}dx=\frac{\pi^2}{3\log(2)}.\]

Next, we show that \eqref{e:equiC} implies \eqref{e:equiC2}. Let
$$G_F(T)=\sum_{\gamma\in \calC'(T)}\int_\gamma F,\;\mbox{ and } H_F(T)=\sum_{\gamma\in \calC'(T)}\frac{1}{\ell(\gamma)}\int_\gamma F.$$
Summation by parts gives
\begin{equation}\label{e:HF}
H_F(T)=\frac{G_F(T)}{T}+\int^T \frac{G_F(t)}{t^2}dt.
\end{equation}
The bound $|\calC'(T)|\lesssim T^\alpha e^{cT}$ implies that $G_1(T)\lesssim T^{\alpha+1}e^{cT}$.
We can thus bound $|G_F(t)|\lesssim G_1(t)\lesssim t^{\alpha+1}e^{ct}$ under the integral in \eqref{e:HF} to get that
\[H_F(T)=\frac{G_F(T)}{T}+O(T^{\alpha-1}e^{cT}).\]
Using the lower bound $|\calC'(T)|\gtrsim T^\beta e^{cT}$ we get that $\frac{H_F(T)}{|\calC'(T)|}\sim \frac{G_F(T)}{T|\calC'(T)|}$ and in particular $\frac{G_1(T)}{T|\calC'(T)|}\sim 1$.
Finally, \eqref{e:equiC} implies that $\frac{G_F(T)}{G_1(T)}\sim \int_{T^1\calM}f$ so that indeed
\[\frac{H_F(T)}{|\calC'(T)|} \sim \frac{G_F(T)}{T|\calC'(T)|} \sim \int_{T^1\calM}f.\]

\end{proof}

\section{Proof of main Theorems}
Theorems \ref{t:main} and \ref{t:inert} follow directly from Theorems \ref{t:Qn} and \ref{t:Qodd} (together with Proposition \ref{p:equi}).
In order to prove Theorems \ref{t:Qn} and \ref{t:Qodd} it is enough to establish $\eqref{e:equiQn}$ and $\eqref{e:Qodd}$ for a dense set of test functions in $C([0,1])$. Let $\calB$ denote the Banach space of holomorphic functions on the disc $D_{3/2}(1)=\{z\in \bbC: |z-1|<3/2\}$ that are continuous on the boundary.
The family of test functions $\{f_{|_{[0,1]}}|f\in \calB\}$ is clearly dense in $C([0,1])$ (as it contains all polynomials).
For these test functions we will follow the approach of \cite{Pollicott86}, that is, we study the analytic continuation of a certain $\eta$-unction and use a suitable Tauberian theorem.

\subsection{The $\eta$-functions}
For a fixed test function $f\in \calB$ and a parameter $\theta\in (-1,1)$ we consider the functions $\eta_{f,\theta}(s)$ and $\eta_f^{\mathrm{odd}}(s)$ given by the series
\begin{equation}\label{e:eta}\eta_{f,\theta}(s)=\sum_{x\in Q}f(x)\sum_{k=1}^\infty 2\cos(\pi k\theta\Alt(x))e^{-k\ell(x)s},\end{equation}
and
\begin{equation}\label{e:etaodd}\eta_{f}^{\mathrm{odd}}(s)=2\sum_{x\in Q_{\mathrm{odd}}}f(x)\mathop{\sum_{k=1}^\infty}_{\mathrm{odd}}e^{-\frac{k\ell(x)s}{2}}.\end{equation}
We will show

\begin{prop}\label{t:etaodd}
The series $\eta_{f}^{\mathrm{odd}}(s)$ absolutely converges for $\Re(s)>1$, has a meromorphic continuation to the half plane $\Re(s)>1/2$ with one simple pole located at $s=1$ and residue
$\mathop{\mathrm{Res}}_{s=1}\eta_{f}^{\mathrm{odd}}=\frac{6\log(2)}{\pi^2}\int_0^1 fd\nu$.
\end{prop}
\begin{prop}\label{t:eta}
The series $\eta_{f,\theta}(s)$ absolutely converges for $\Re(s)>1$, has a meromorphic continuation to the half plane $\Re(s)>1/2$ with at most one simple pole in the half plane $\Re(s)>3/4$ located at $s_\theta=1-3|\theta|$ (when $|\theta|<1/12$). Moreover, in this case the residue
$$R_f(\theta)=\mathop{\mathrm{Res}}_{s=s_\theta}\eta_{f,\theta},$$
satisfies
$$R_f(\theta)=\frac{6\log(2)}{\pi^2}\int_0^1 fd\nu+O(|\theta|^{1/2}).$$
\end{prop}

\begin{rem}
With a little more work one can replace the error term of $O(|\theta|^{1/2})$ by
$O(|\theta||\log\tfrac{1}{|\theta|}|^2)$, but for our purpose this weaker bound will be sufficient.
\end{rem}

\begin{rem}
For $\theta=0$, the equality $R_f(0)=\frac{6\log(2)}{\pi^2}\int_0^1 fd\nu$ follows directly from \cite[Proposition 2]{Pollicott86}.
For $\theta\neq 0$ and the specific test function $f_0(x)=-2\log(x)$, 
Proposition \ref{t:eta} follows from \eqref{e:SM}
with an exact formula $R_{f_0}(\theta)=1=\frac{6\log(2)}{\pi^2}\int_0^1 f_0d\nu$ (see the proof of Proposition \ref{p:poles}). The proof of the general case will follow from these two cases together with a perturbation theory argument.
\end{rem}

\subsection{The Tauberian Theorem}
We postpone the proof of Propositions \ref{t:etaodd} and \ref{t:eta} to the following sections. We will now show how to use them together with an appropriate Tauberian theorem to prove Theorems \ref{t:Qn} and \ref{t:Qodd}. Specifically, we will need the following version of the Wienner-Ikehara Tauberian Theorem (see \cite[\S III, Theorem 4.2]{Korevaar04}).
\begin{thm*}[Wiener-Ikehara]
Let $S(t)$ vanish for $t<0$, be nondecreasing, continuous from the right, and such that the Laplace transform $g(s)=\int_0^\infty e^{-ts}dS(t)$ exists for $\Re(s)>1$. Suppose that for some constant $A$ and for any $\lambda>0$ the function
$g_x(y)=g(x+iy)-\frac{A}{x+iy-1}$ converges as $x\to 1^+$ to some limit function  $g_1(y)$ in $L^1(-\lambda,\lambda)$. Then $e^{-t}S(t)\to A$ as $t\to \infty$.
\end{thm*}
\begin{proof}[Proof of Theorem \ref{t:Qn}]
We will apply the Wiener-Ikehara Tauberian theorem for $S(t)=S_{f,n}(t)$ given by
\begin{equation}S(t)=\sum_{k|n}k\mathop{\sum_{x\in Q^+_{n/k}}}_{\ell(x)\leq t/k} f(x)\ell(x).\end{equation}
By adding to $f$ a sufficiently large multiple of $f_0(x)=-2\log(x)$ we can insure that $S(t)$ is nondecreasing.

The Laplace transform of $S(t)$ is given by
\[g(s)=\sum_{k|n}k\mathop{\sum_{x\in Q^+_{n/k}}} f(x)\ell(x)e^{-k\ell(x)s},\]
which absolutely converges for $\Re(s)>1$.
We can obtain $g(s)$ by integrating the derivative of $\eta_{f,\theta}$ against $e^{-i\pi n\theta}$, specifically we have
\begin{equation}\label{e:inteta}\int_{-1}^1\eta'_{f,\theta}(s)e^{-i\pi n\theta}d\theta=\left\lbrace\begin{array}{cc} -2g(s) & n\neq 0\\ -4g(s) & n=0\end{array}\right..\end{equation}
Using Proposition \ref{t:eta} we can write
\[\eta'_{f,\theta}(s)=\frac{-R_f(\theta)}{(s-1+3|\theta|)^2}+\phi(s),\]
with $\phi(s)$ holomorphic in $\Re(s)>3/4$ and $R_f(\theta)=0$ for $|\theta|>1/12$. Plugging this back in \eqref{e:inteta} we get
\[g(s)=c\int_{-1/12}^{1/12}\frac{R_f(\theta)e^{-i\pi n\theta}}{(s-1+3|\theta|)^2}d\theta+ \Phi(s),\]
with $\Phi(s)$ holomorphic in $\Re(s)>3/4$ and $c=\left\lbrace\begin{array}{cc}1/2 & n\neq 0\\ 1/4 & n=0\end{array}\right.$.

Taking $A_f=\frac{R_f(0)}{3}$ (respectively $\frac{R_f(0)}{6}$ if $n=0$), we have
\begin{eqnarray*}g_x(y)&=&g(x+iy)-\frac{A_f}{x+iy-1}\\&=&c\int_{-1/12}^{1/12}\frac{R_f(\theta)e^{-i\pi n\theta}-R_f(0)}{(iy+x-1+3|\theta|)^2}d\theta+\Phi(x+iy).\end{eqnarray*}

The bound
\[|R_f(\theta)e^{-i\pi n\theta}-R_f(0)|\leq |R_f(\theta)-R_f(0)|+|R_f(0)||e^{in\theta}-1|=O(\sqrt{\theta}),\]
implies that the function
\[g_1(y)=c\int_{-1/12}^{1/12}\frac{R_f(\theta)e^{-i\pi n\theta}-R_f(0)}{(iy+3|\theta|)^2}d\theta+\Phi(1+iy),\]
is in $L^1(-\lambda,\lambda)$ for any $\lambda>0$ and that $g_x\to g_1$ in $L^1(-\lambda,\lambda)$.
The Wiener-Ikehara Tauberian theorem now implies that
\[\sum_{k|n}k\mathop{\sum_{x\in Q^+_{n/k}}}_{\ell(x)<T/k} f(x)\ell(x)\sim A_fe^{T}.\]
Since the contribution of all $k>1$ to this sum is bounded by $O(T^2e^{T/2})$ we get
\[\sum_{x\in Q^+_{n}(T)} f(x)\ell(x)\sim A_fe^{T},\]
and summation by parts gives
\[\sum_{x\in Q_n^+(T)} f(x)\sim A_f\frac{e^{T}}{T}.\]

In particular, taking $f=1$ we get that
$|Q_n^+(T)|\sim A_1\frac{e^{T}}{T}$ and since $|Q_n(T)|=|Q_{-n}(T)|$ (as the Gauss map gives a bijection between them) we get that
\[|Q_n(T)|\sim  \frac{R_1(0)}{6}\frac{e^T}{T}=\frac{\log(2)}{\pi^2}\frac{e^T}{T}.\]
Finally, for any other $f\in \calB$,
\[\lim_{T\to\infty}\frac{1}{|Q_n^+(T)|}\sum_{x\in Q_n^+(T)} f(x)= \frac{A_f}{A_1}=\frac{R_f(0)}{R_1(0)}=\int_0^1 fd\nu.\]
\end{proof}
\begin{proof}[Proof of Theorem \ref{t:Qodd}]
From Proposition \ref{t:etaodd} we can write
\[\eta_f^{\mathrm{odd}}(s)=\frac{\frac{6\log(2)}{\pi^2}\int fd\nu}{1-s}+\phi(s),\]
with $\phi(s)$ holomorphic in $\Re(s)>\frac{1}{2}$.
The Wiener-Ikehara Tauberian theorem now implies that
\[2 \sum_{k\;\mathrm{odd}}\mathop{\sum_{y\in Q_{\mathrm{odd}}}}_{\frac{k\ell(x)}{2}\leq T}f(y)\sim \frac{6\log(2)}{\pi^2}\left(\int fd\nu\right) e^T,\]
and as before we can ignore the contribution of $k>1$ to get
\[\mathop{\sum_{y\in Q_{\mathrm{odd}}}}_{\ell(x)\leq T}f(y)\sim \frac{3\log(2)}{\pi^2}\left(\int fd\nu\right) e^{T/2}.\]
 \end{proof}

\subsection{The transfer operator}
In order to obtain the analytic continuation of $\eta_{f,\theta}$ and $\eta_f^{\mathrm{odd}}$, we relate them to the Fredholm determinant of a suitable Ruelle-Perron-Frobenius transfer operator.

Fix a test function $f\in \calB$. For any complex numbers $s,w\in \bbC$ with $\Re(s)>\tfrac{1}{2}$ let $\chi_{s,w}(z)=z^{2s}e^{wf(z)}$.
For any $\theta\in [-1,1]$ we define the Ruelle-Perron-Frobenius operator $\calL^\theta_{s,w}:\calB\to \calB$ by
\begin{equation}\label{e:transfer}\calL^\theta_{s,w}g(z)=\sum_{a=1}^\infty e^{i\pi a \theta }\chi_{s,w}(\frac{1}{a+z}) g(\frac{1}{a+z}),\end{equation}
and we denote by
\[\calT^\theta_{s,w}=\calL^\theta_{s,w}\calL^{-\theta}_{s,w}.\]
For small $\theta$ (and $w$) we think of $\calL^\theta_{s,w}$ as a perturbation of $\calL^0_{s,w}$ (and $\calL^0_{s,0}$) studied by Pollicott \cite{Pollicott86} and Mayer \cite{Mayer76,Mayer91}.

The same arguments as in \cite{Mayer76} show that this operator is a nuclear operator.
Specifically we show the following.
\begin{prop}
For $\Re(s)>\tfrac{1}{2}$ and any $n\in \bbN$
\begin{eqnarray*}\lefteqn{\Tr\left((\calT^\theta_{s,w})^n\right)=}\\
&& \sum_{\bfa\in \bbN^{2n}}e^{i\pi\theta\Alt(\bfa)}\frac{e^{-s\ell(\bfa)}}{1-e^{-\ell(\bfa)}}\exp\big(w\sum_{j=1}^{2n}f(\T^j[\overline{a_1,\ldots,a_{2n}}])\big)
\end{eqnarray*}
where $\Alt(\bfa)=\sum_{j=1}^{2n}(-1)^j a_j$ and $\ell(\bfa)=-2\sum_{j=1}^{2n} \log(\T^j[\overline{a_1,\ldots,a_{2n}}])$.
(Note that the expansion $[\overline{a_1,\ldots,a_{2n}}]$ here is not necessarily minimal.)
\end{prop}
\begin{proof}
For any $\bfa\in \bbN^{n}$ consider the operator $L_\bfa :\calB\to\calB$ given by
\[L_\bfa g(z)=g([a_1,\ldots, a_{n};z])\prod_{j=1}^{n}\chi_{s,w}([a_j,\ldots,a_{n};z]),\]
where we use the notation
\[[a_1,\ldots,a_{n};z]=\tfrac{1}{a_1+\frac{1}{\cdots+\frac{1}{a_{n}+z}}}.\]
In \cite[Page 199]{Mayer76}, Mayer computed the spectrum of the operators $L_{\bfa}$ and showed that their trace is given by
\begin{eqnarray}\label{e:La}\Tr(L_{\bfa})&=&\frac{\prod_{j=1}^{n}\chi_{s,w}(\T^{j}[\overline{a_1,\ldots,a_{n}}])}{1-(-1)^n\prod_{j=1}^{n}(\T^j[\overline{a_1,\ldots,a_{n}}])^2}\\
\nonumber &=&\frac{e^{-s\ell(\bfa)}}{1-(-1)^ne^{-\ell(\bfa)}}\exp(w\sum_{j=1}^{n}f(\T^j[\overline{a_1,\ldots,a_{n}}])).
\end{eqnarray}
For $\Re(s)>\tfrac{1}{2}$ all the sums absolutely converge and we can expand
\[(\calL^\theta_{s,w}\calL^{-\theta}_{s,w})^n=\sum_{\bfa\in \bbN^{2n}}e^{i\pi\theta\Alt(\bfa)}L_{\bfa},\]
hence,
\[\Tr\left((\calT^\theta_{s,w})^n\right)=\sum_{\bfa\in \bbN^{2n}}e^{i\pi\theta\Alt(\bfa)}\Tr(L_{\bfa}),\]
concluding the proof.
\end{proof}
Using the theory of Fredholm determinants for nuclear operators on Banach space we get
\begin{cor}
The function $Z_\theta(s,w)=\det(1-\calT^\theta_{s,w})$ is holomorphic in $w,s$ for $\Re(s)>\tfrac{1}{2}$ and is non-zero unless $1$ is an eigenvalue for $\calT^\theta_{s,w}$. Moreover, for $\Re(s)>1$ it is given by
\begin{eqnarray*}
\lefteqn{Z_\theta(s,w)=}\\
&& \exp\left(-\sum_{n=1}^\infty \frac{1}{n}\sum_{\bfa\in \bbN^{2n}}e^{i\pi\theta\Alt(\bfa)}\frac{e^{-s\ell(\bfa)}}{1-e^{-\ell(\bfa)}}\exp\big(w\sum_{j=1}^{2n}f(\T^j[\overline{a_1,\ldots,a_{2n}}])\big)\right).
\end{eqnarray*}
\end{cor}

Next we consider the logarithmic derivatives
\begin{equation} \tilde\eta_{f,\theta}(s)=-\pd{\log(Z_\theta)}{w}(s,0),\end{equation}
and note that it is closely related to the function $\eta_{f,\theta}$ defied in \eqref{e:eta}.
\begin{prop}
$\eta_{f,\theta}(s)=\tilde\eta_{f,\theta}(s)-\tilde\eta_{f,\theta}(s+1)$.
\end{prop}
\begin{proof}
For $\Re(s)>1$ we have
\begin{eqnarray*}
\lefteqn{\tilde\eta_{f,\theta}(s)-\tilde\eta_{f,\theta}(s+1)=}\\
&&\sum_{n=1}^\infty \frac{1}{n}\sum_{\bfa\in \bbN^{2n}}e^{i\pi\theta\Alt(\bfa)}e^{-s\ell(\bfa)}\left(\sum_{j=1}^{2n}f(\T^j[\overline{a_1,\ldots,a_{2n}}])\right).
\end{eqnarray*}
Here the expansions are not minimal, in particular, for any $\bfa\in \bbN^{2n}$ if $x=[\overline{a_1,\ldots,a_{2n}}]$ and $Q_x$ is the orbit of $x$ under the Gauss map, then $2n=k|Q_x|$ for some $k\in \bbN$, $\ell(\bfa)=k\ell(x)$, and $\Alt(\bfa)=k\Alt(x)$. We can thus rewrite this as
\begin{eqnarray*}
\tilde\eta_{f,\theta}(s)-\tilde\eta_{f,\theta}(s+1)=\sum_{x\in Q}\sum_{k=1}^\infty e^{i\pi\theta k\Alt(x)}e^{-sk\ell(x)}\frac{2}{|Q_x|}\sum_{y\in Q_x}f(y).
\end{eqnarray*}
For any $y=\T^j x\in Q_x$ we have $\ell(y)=\ell(x)$, $Q_x=Q_y$ and $\Alt(y)=(-1)^j\Alt(x)$. Changing the order of summation we get
\begin{eqnarray*}
\tilde\eta_{f,\theta}(s)-\tilde\eta_\theta(s+1)&=&\sum_{y\in Q}\sum_{k=1}^\infty e^{-sk\ell(y)}f(y)\frac{2}{|Q_y|}\sum_{x\in Q_y}e^{\pi i k \theta \Alt(x)}\\
&=&\sum_{y\in Q}\sum_{k=1}^\infty e^{-sk\ell(y)}f(y)2\cos(\pi k\theta \Alt(y))= \eta_{f,\theta}(s)
\end{eqnarray*}
\end{proof}

\subsection{Proof of Proposition \ref{t:eta}}
Since $\eta_{f,\theta}(s)=\tilde\eta_{f,\theta}(s)-\tilde\eta_{f,\theta}(s+1)$, the following establishes the first part of Proposition \ref{t:eta}.
\begin{prop}\label{p:poles}
$\tilde{\eta}_{f,\theta}(s)$  has a meromorphic continuation to $\Re(s)>\frac{1}{2}$ with at most one simple pole in $\Re(s)>\tfrac{3}{4}$ located at $s_\theta=1-3|\theta|$ when $|\theta|< \tfrac{1}{12}$ and no poles in $\Re(s)>\tfrac{3}{4}$ otherwise.
\end{prop}
\begin{proof}
Since $Z_\theta(s,w)$ is analytic in $\Re(s)>\frac{1}{2}$ we get that $\tilde{\eta}_{f,\theta}(s)$ is meromorphic there, and it can have a pole only at values of $s$ for which $1$ is an eigenvalue of $\calT^\theta_{s,0}$ (notice that this no longer depends on the test function $f$).
Next, note that $1$ is a (simple) eigenvalue of $\calT^\theta_{s,0}$ if and only if $s$ is a (simple) zero of $Z_\theta(s,0)$.

Now, for $\Re(s)>1$ we have
$Z_\theta(s,0)= \exp\left(-\tilde{\zeta}_\theta(s)\right)$,
with
\begin{eqnarray*}\tilde{\zeta}_\theta(s)&=&\sum_{n=1}^\infty \frac{1}{n}\sum_{\bfa\in \bbN^{2n}}e^{i\pi\theta\Alt(\bfa)}\frac{e^{-s\ell(\bfa)}}{1-e^{-\ell(\bfa)}}\\
\end{eqnarray*}
Let $\zeta_\theta(s)=\tilde{\zeta}_\theta(s)-\tilde{\zeta}_\theta(s+1)$ and note that, using the correspondence between closed geodesics and quadratic irrationals, we can also write this as
\[\zeta_\theta(s)=\sum_{\gamma} e^{i\theta \Psi(A_\gamma)}e^{-s\ell(\gamma)},\]
where the sum is over all closed geodesics (not just primitive).

Let $\Lambda_\theta(T)=\sum_{\ell(\gamma)\leq T}\ell(\gamma)e^{i\pi\theta\Psi(A_\gamma)}$. Then for $\Re(s)>1$ we can write
\begin{eqnarray*}
\zeta'_\theta(s)-\frac{s}{s_\theta(s-s_\theta)}&=& -\int_0^\infty e^{-st}d\Lambda_\theta(t)dt+\frac{s}{s_\theta}\int_0^\infty e^{(s_\theta-s)t}dt\\
&=& -s\int_0^\infty e^{-st}\Lambda_\theta(t)dt+\frac{s}{s_\theta}\int_0^\infty e^{(s_\theta-s)t}dt\\
&=& s\int_0^\infty e^{-st}\big(\frac{e^{s_\theta t}}{s_\theta}-\Lambda_\theta(t)\big)dt.
\end{eqnarray*}
Now \eqref{e:SM} implies that $|\frac{e^{s_\theta t}}{s_\theta}-\Lambda_\theta(t)|=O(e^{3t/4})$ so the integral on the right absolutely converges for $\Re(s)>3/4$, and hence defines an analytic function there.

We thus see that $\zeta'_\theta(s)$ (and hence also $\tilde{\zeta}_\theta'(s)$) has a meromorphic continuation to
the half plane $\Re(s)>3/4$ with a single simple pole at $s=s_\theta$ and residue $\mathrm{Res}_{s=s_\theta}\zeta_\theta'=1$ (when $s_\theta>3/4$).
Consequently, $Z_\theta(s,0)=\exp(-\tilde{\zeta}_\theta(s))$ does not vanish in $\Re(s)>\frac{3}{4}$ except for one simple zero at $s_\theta=1-3|\theta|$.
\end{proof}

For any $\theta\in(-\tfrac{1}{12},\tfrac{1}{12})$ we have
$$R_f(\theta)=\mathrm{Res}_{s=s_\theta}\eta_{f,\theta}=\mathrm{Res}_{s=s_\theta}\tilde{\eta}_{f,\theta}.$$
Since \cite[Proposition 2]{Pollicott86} implies that $R_f(0)=\frac{6\log(2)}{\pi^2}\int_0^1 fd\nu$, the following proposition gives the second half of Proposition \ref{t:eta}.
\begin{prop}\label{p:modulus}
There is some $\delta>0$ such that
$$R_f(\theta)=R_f(0)+O(|\theta|^{1/2}),$$
uniformly for $\theta\in[-\delta,\delta]$
\end{prop}
For the proof we will use some perturbation theory of compact operators (see e.g.  \cite{Kato82} for some of the standard arguments). In particular, we will need the following estimates for the modulus of continuity of the operator $\calT^\theta_{s,w}$ with respect to the $\theta$ parameter, as well as the modulus of its partial derivatives.
%
%
%

\begin{lem}\label{l:modulus}
For $\theta$ sufficiently small we can bound $\norm{\pd{}{w}_{|_{w=0}}(\calT^\theta_{s,w}-\calT^0_{s,w})}$,
$\norm{\calT^\theta_{s,0}-\calT^0_{s,0}}$, and $\norm{\pd{}{s}(\calT^\theta_{s,0}-\calT^0_{s,0})}$, by $O(\sqrt{|\theta|})$
uniformly in any half plane $\Re(s)\geq c> 3/4$.
\end{lem}
\begin{proof}
Since $\calT^\theta_{s,w}=\calL^\theta_{s,w}\calL^{-\theta}_{s,w}$ it is enough to prove these bounds for $\calL^{\theta}_{s,w}$.
For $\Re(s)\geq c>3/4$ for any $g\in \calB$ with $\norm{g}=1$ we can bound
\begin{eqnarray*}
|\pd{}{s}(\calL^\theta_{s,0}-\calL^0_{s,0})g(z)|&=&\left|\sum_{a=1}^\infty (e^{i\theta a}-1) g(\frac{1}{a+z})\log(\frac{1}{a+z})(\frac{1}{a+z})^{2s}\right|\\
&\lesssim & \norm{g}\sum_{a=1}^\infty|e^{i\theta a}-1|a^{-3/2}\\
&\lesssim & |\theta| \sum_{a\leq |\theta|^{-1}} a^{-1/2}+\sum_{a\geq  |\theta|^{-1}} a^{-3/2}\lesssim  |\theta|^{1/2}.
\end{eqnarray*}
Since this holds for all $z\in D_{3/2}(1)$ and all $g\in \calB$ with $\norm{g}=1$ indeed
$\norm{\pd{}{s}(\calL^\theta_{s,0}-\calL^0_{s,0})}\lesssim |\theta|^{1/2}$.
The arguments for $\calL^\theta_{s,0}$ and $\pd{}{w}\calL^\theta_{s,w}$ are identical.
\end{proof}
\begin{proof}[Proof of proposition \ref{p:modulus}]
Since $Z_\theta(s,0)$ has a simple zero at $s_\theta$, the operator $\calT^\theta_{s_\theta,0}$ has $\lambda=1$ as a simple isolated eigenvalue.
Thinking of $\calT^\theta_{s,w}$ as a perturbation $\calT^\theta_{s_\theta,0}$ which is analytic in $s,w$ and continuous in $\theta$, we find that there is a function $\lambda_\theta(s,w)$ (analytic in $s,w$ and continuous in $\theta$) such that for all sufficiently small $\theta$, the function
$\tilde{Z}_\theta(s,w)=\frac{Z_\theta(s,w)}{1-\lambda_\theta(s,w)}$ is analytic and non zero in some neighborhood of $(s,w)=(s_\theta,0)$.
We can now write
$$\tilde{\eta}_\theta(s)=\frac{-1}{1-\lambda_\theta(s,0)}\pd{\lambda_\theta}{w}(s,0)-\frac{\pd{\tilde{Z}_\theta}{w}(s,0)}{\tilde{Z}_\theta(s,0)}.$$
and since  $\frac{\pd{\tilde{Z}_\theta}{w}(s,0)}{\tilde{Z}_\theta(s,0)}$ is analytic in some neighborhood of $s=s_\theta$ we get that
\[R_f(\theta)=\frac{\pd{\lambda_\theta}{w}(s_\theta,0)}{\pd{\lambda_\theta}{s}(s_\theta,0)}.\]

Next, from continuity of $\calT_{s,w}^\theta$ (in all parameters), for $|s-1|,|w|,|\theta|$ sufficiently small the projection operator to the (one dimensional) eigenspace of $\lambda_\theta(s,w)$  is given by
\[\calP^\theta_{s,w}=\frac{1}{2\pi i}\int_{|\xi-1|=r_0}(\calT^\theta_{s,w}-\xi)^{-1}d\xi,\]
for some fixed small radius $r_0$ (for which $\lambda_\theta(s,w)$ is the only eigenvalue of $\calT_{s,w}^\theta$ satisfying $|\lambda-1|<r_0$).
Then, for any $g\in \calB,\;z\in D_{3/2}(1)$ and $\theta, w,s-1$ sufficiently small we have
\[\calT^\theta_{s,w}\calP^\theta_{s,w}g(z)=\lambda_\theta(s,w)\calP^\theta_{s,w}g(z).\]

In particular, let $g_0\in \calB$ denote the normalized eigenfunction of $\calT^0_{1,0}$ with eigenvalue $1$, and $z_0\in D_{3/2}(1)$ any point satisfying $g_0(z_0)>\tfrac{1}{2}$. Lemma \ref{l:modulus} implies that $\norm{\calP^\theta_{s,w}-\calP^0_{1,0}}=O(|\theta|^{1/2}+|w|+|s-1|)$ so that for $\theta, w,s-1$ sufficiently small $\calP^\theta_{s,w}g_0(z_0)\neq 0$ and we have
\begin{equation}
\label{e:lambda}\lambda_\theta(s,w)=\frac{\calT^\theta_{s,w}\calP^\theta_{s,w}g_0(z_0)}{\calP^\theta_{s,w}g_0(z_0)}.
\end{equation}
Taking derivatives of \eqref{e:lambda} with respect to $w$ and $s$ respectively 
and using Lemma \ref{l:modulus}, we get that
 $$|\pd{\lambda_\theta}{s}(s_\theta,0)-\pd{\lambda_0}{s}(s_\theta,0)|=O(\sqrt{|\theta|}),$$
 and similarly
$$|\pd{\lambda_\theta}{w}(s_\theta,0)-\pd{\lambda_0}{w}(s_\theta,0)|=O(\sqrt{|\theta|}).$$
Also, since $\calT^0_{s,w}$ (and hence $\lambda_0(s,w)$) is analytic in $s$ and $w$ we have
$$\pd{\lambda_0}{s}(s_\theta,0)=\pd{\lambda_0}{s}(1,0)+O(|s_\theta-1|)=\pd{\lambda_0}{s}(1,0)+O(|\theta|),$$
and similarly
$$\pd{\lambda_0}{w}(s_\theta,0)=\pd{\lambda_0}{w}(1,0)+O(|\theta|).$$
Combinimg these estimates we get that
\begin{eqnarray*}
\pd{\lambda_\theta}{w}(s_\theta,0)=\pd{\lambda_0}{w}(1,0)+O(\sqrt{|\theta|}),\\ \pd{\lambda_\theta}{s}(s_\theta,0)=\pd{\lambda_0}{s}(1,0)+O(\sqrt{|\theta|}),\end{eqnarray*}
and hence
\[R_f(\theta)=\frac{\pd{\lambda_\theta}{w}(s_\theta,0)}{\pd{\lambda_\theta}{s}(s_\theta,0)}=\frac{\pd{\lambda_0}{w}(1,0)}{\pd{\lambda_0}{s}(1,0)}+ O(\sqrt{|\theta|})=R(0)+O(\sqrt{|\theta|}).\]
\end{proof}
\subsection{Proof of Propositoin \ref{t:etaodd}}
Fix a test function $f\in \calB$ and consider the transfer opertor $\calL_{s,w}=\calL_{s,w}^0$ as in \eqref{e:transfer}.
Define two Zeta function by
\[Z_{f}^\pm(s,w)=\det(1\pm\calL_{s,w}),\]
and let
\[Z^{\mathrm{odd}}_{f}(s,w)=\frac{Z_{f}^-(s,w)}{Z_{f}^+(s,w)}.\]
This function is meromorphic for $\Re(s)>\frac{1}{2}$ and
using \eqref{e:La} to compute $\Tr\big((\calL_{s,w})^n\big)$ we get that for $\Re(s)>1$ it is given by
\[Z^{\mathrm{odd}}_{f}(s,w)=\exp\left(-2\mathop{\sum_{n=1}^\infty}_{\mathrm{odd}}\frac{1}{n}\sum_{\bfa\in\bbN^n}
\frac{e^{-s\ell(\bfa)}}{1+e^{-\ell(\bfa)}}\exp\big(w\sum_{j=1}^n f(\T^j[\overline{a_1,\ldots,a_n}])\big)\right).\]
For $\bfa\in \bbN^n$ with $n$ odd we have that $x=[\overline{a_1,\ldots,a_n}]$ is in $Q_{\mathrm{odd}}$ (because its minimal expansion must also be odd). Moreover in this case $n=k|Q_x|$ for some odd $k\in \bbN$ and $\ell(\bfa)=\frac{k\ell(x)}{2}$ (recall that $\ell(x)$ is computed using the minimal even expansion). As before, if $\tilde{\eta}^{\mathrm{odd}}_f(s)=-\pd{}{w}_{|_{w=0}}\log(Z^{\mathrm{odd}}_f(s,w))$ then
\begin{eqnarray*}
\tilde{\eta}^{\mathrm{odd}}_f(s)+\tilde{\eta}^{\mathrm{odd}}_f(s+1) &=& 2\sum_{x\in Q_{\mathrm{odd}}}\sum_{k\;\mathrm{odd}}e^{-sk\ell(x)/2}\frac{1}{|Q_x|}\sum_{y\in Q_x} f(y)\\
&=&2 \sum_{y\in Q_{\mathrm{odd}}}f(y)\sum_{k\;\mathrm{odd}}e^{-sk\ell(x)/2}=\eta_f^{\mathrm{odd}}(s)
\end{eqnarray*}

As in the proof of Proposition \ref{p:poles}, we see that $\eta_f^{\mathrm{odd}}(s)$ has a meromorphic continuation to $\Re(s)>\frac{1}{2}$ and its only poles are at the zeros and poles of
$$Z^{\mathrm{odd}}_{f}(s,0)=\frac{\det(1-\calL_{s})}{\det(1+\calL_{s})}=\frac{Z^-(s)}{Z^+(s)}.$$

We recall that the Selberg Zeta function,
\[Z(s)=\prod_{\gamma\in \calC}\prod_{k=1}^\infty (1-e^{-(s+k)\ell(\gamma)}),\]
has a simple zero at $s=1$ and no other zeros in the half plane $\Re(s)>\frac{1}{2}$.
In \cite{Mayer91}, Mayer showed that the Selberg Zeta function has the factorization
$Z(s)=Z^-(s)Z^+(s)$ (this factorization was further studied in \cite{Efrat93} and is related to odd and even spectrum of the Laplacian).
Consequently, $Z^{\mathrm{odd}}_{f}(s,0)=\frac{Z^-(s)}{Z^+(s)}$ has an analytic continuation to $\Re(s)>\frac{1}{2}$ with no zeros except a simple zero at $s=1$ (it was also shown in \cite{Mayer76} that $1$ is an eigenvalue of $\calL_1$ so that $Z^-(1)=0$).

We thus see that $\eta^{\mathrm{odd}}_f(s)$ has a meromorphic continuation to $\Re(s)>\frac{1}{2}$ with a single simple pole at $s=1$.
Finally, let $\lambda_0(s,w)$ denote the eigenvalue of $\calL_{s,w}$ obtained as the analytic continuation of the simple eigenvalue $\lambda_0(1,0)=1$ of $\calL_{1,0}$. Then, as before we have that
\[\mathop{\mathrm{Res}}_{s=1}\eta_f^{\mathrm{odd}}(s)=\frac{\pd{\lambda_0}{w}(1,0)}{\pd{\lambda_0}{s}(1,0)},\]
which by \cite[Proposition 2]{Pollicott86} is equal to $\frac{6\log(2)}{\pi^2}\int fd\nu$.


\begin{rem}
In this proof we used the fact that the Selberg Zeta function has no zeros in $\Re(s)>\tfrac{1}{2}$ except at $s=1$, which follows from the Selberg trace formula. We note that all we actually need in order to prove Theorem \ref{t:Qodd} is that it has no other zeros on the line $\Re(s)=1$, which can be proved using only the dynamical approach of \cite{Mayer76}.
\end{rem}


\def\cprime{$'$} \def\cprime{$'$}
\providecommand{\bysame}{\leavevmode\hbox to3em{\hrulefill}\thinspace}
\providecommand{\MR}{\relax\ifhmode\unskip\space\fi MR }
\providecommand{\MRhref}[2]{%
  \href{http://www.ams.org/mathscinet-getitem?mr=#1}{#2}
}
\providecommand{\href}[2]{#2}

\end{document}